\documentclass[12pt,letterpaper,reqno]{amsart}
\usepackage{amsmath,amsthm,amsfonts,amssymb,mathtools}

\newtheorem{theorem}{Theorem}
\newtheorem{proposition}[theorem]{Proposition}

\newtheorem{lemma}[theorem]{Lemma}

\numberwithin{equation}{section}

\begin{document}

\title[Quantitative norm convergence of double ergodic averages]{Quantitative norm convergence of double ergodic averages
associated with two commuting group actions}

\author{Vjekoslav Kova\v{c}}
\address{Department of Mathematics, Faculty of Science, University of Zagreb, Bijeni\v{c}ka cesta 30, 10000 Zagreb, Croatia}
\email{vjekovac@math.hr}
\thanks{This work has been fully supported by the Croatian Science Foundation under the project 3526.}

%\keywords{mean ergodic theorem, multiple ergodic averages, norm convergence, Cantor group, multilinear operator, Calder\'{o}n-Zygmund operator}

\begin{abstract}
We study double averages along orbits for measure preserving actions of $\mathbb{A}^\omega$,
the direct sum of countably many copies of a finite abelian group $\mathbb{A}$.
In this article we show an $\textup{L}^p$ norm-variation estimate for these averages,
which in particular reproves their convergence in $\textup{L}^p$
for any finite $p$ and for any choice of two $\textup{L}^{\infty}$ functions.
The result is motivated by recent questions on quantifying convergence of multiple ergodic averages.
\end{abstract}

\maketitle

\section{Introduction}

Let $(X,\mathcal{F},\mu)$ be a probability space and let $T_1,T_2,\ldots,T_r\colon X\to X$
be mutually commuting measure $\mu$ preserving transformations.
Multiple averages
\begin{equation}\label{eqmultipleziaverages}
\textup{M}_n(f_1,f_2,\ldots,f_r) := \frac{1}{n}\sum_{k=0}^{n-1}(f_1\circ T_1^k)(f_2\circ T_2^k)\cdots(f_r\circ T_r^k)
\end{equation}
for $f_1,f_2,\ldots,f_r\in\textup{L}^\infty$ were motivated by the work of Furstenberg and Katznelson \cite{FurstenbergKatznelson78}
on multidimensional Szemer\'{e}di's theorem and have attracted much attention in the literature.
Their convergence in the $\textup{L}^2$ norm as $n\to\infty$ was first established by Tao \cite{Tao08}.
For $r=1$ this was shown much earlier by von Neumann \cite{Neumann32} and for $r=2$ by Conze and Lesigne \cite{ConzeLesigne84}.
On the other hand, far reaching generalizations were later given by Austin \cite{Austin10} and Walsh \cite{Walsh12}.

However, it is an interesting problem to quantify the convergence of the sequence \eqref{eqmultipleziaverages} by controlling the
number of its jumps of a certain size.
Such a result in the case of a single transformation (i.e.\@ $r=1$) is the following norm-variation estimate,
\begin{equation}\label{eqlinearresult}
\sup_{n_0<n_1<\cdots<n_m} \sum_{j=1}^{m} \big\|\textup{M}_{n_{j-1}}(f)-\textup{M}_{n_{j}}(f)\big\|^{p}_{\textup{L}^{p}(X,\mathcal{F},\mu)}
\,\leq\, C_{p}\, \|f\|^{p}_{\textup{L}^{p}(X,\mathcal{F},\mu)} ,
\end{equation}
which holds for $p\geq 2$ with some constant $C_p$ depending only on the exponent $p$.
The supremum is taken over all positive integers $m$ and all increasing choices of positive integers $n_0,n_1,\ldots,n_m$.
Inequality \eqref{eqlinearresult} was first proved by Jones, Ostrovskii, and Rosenblatt \cite{JonesOstrovskiiRosenblatt96} in the case $p=2$
and generalized by Jones, Kaufman, Rosenblatt, and Wierdl \cite{JonesKaufmanRosenblattWierdl98} to $p\geq 2$.
Avigad and Rute \cite{AvigadRute13} interpreted \eqref{eqlinearresult} by observing that for any $\varepsilon>0$ the sequence
$(\textup{M}_n(f))_{n=1}^{\infty}$ has $O\big(\varepsilon^{-p}\|f\|^p_{\textup{L}^p}\big)$ jumps
of size at least $\varepsilon$ in the $\textup{L}^p$ norm.
They also studied analogues of \eqref{eqlinearresult} in the setting of more general Banach spaces.

Furthermore, Avigad and Rute \cite[\S 7]{AvigadRute13} asked if it was possible to show any quantitative estimates in this direction
for multiple averages \eqref{eqmultipleziaverages} when $r\geq 2$.
So far the only quantitative results exist when $r=2$ and the transformation $T_2$ is an integer power of $T_1$,
and can be found in the work of Bourgain \cite{Bourgain90} and Demeter \cite{Demeter07}; also see a variational estimate
by Do, Oberlin, and Palsson \cite{DoOberlinPalsson12}.
It is also worth noting that papers \cite{Bourgain90}, \cite{Demeter07}, \cite{JonesKaufmanRosenblattWierdl98}
are actually dealing with estimates that establish pointwise convergence a.e., which is significantly stronger than convergence in the $\textup{L}^p$ norm.

We are able to make a slight progress on this question in the particular case $r=2$ by adapting the technique that originated in
the papers \cite{Kovac11}, \cite{Kovac12}, and \cite{KovacThiele13} by Thiele and the author.
However, due to several technical reasons explained later we need to work in a slightly different group setting,
borrowed for instance from \cite{BergelsonTaoZiegler13}.

Let us fix a finite abelian group $\mathbb{A}$ and consider the group $\mathbb{A}^\omega$ defined as the direct sum
$$ \mathbb{A}^\omega := \bigoplus_{k=0}^{\infty}\mathbb{A} = \mathbb{A}\oplus\mathbb{A}\oplus\cdots $$
of countably many copies of $\mathbb{A}$.
Thus, elements of $\mathbb{A}^\omega$ are simply sequences $a=(a_k)_{k=0}^{\infty}$ of elements from $\mathbb{A}$
such that $a_k=0$ if the index $k$ is large enough.
The most natural \emph{F{\o}lner sequence} \cite{Folner55} for $\mathbb{A}^\omega$ is
$$ \Phi_n = \big\{(a_k)_{k=0}^{\infty}\in\mathbb{A}^\omega : a_k=0 \text{ for }k\geq n\big\} \cong \mathbb{A}^n . $$
Furthermore, let $S=(S^a)_{a\in\mathbb{A}^\omega}$ and $T=(T^a)_{a\in\mathbb{A}^\omega}$ be two commuting measure
preserving $\mathbb{A}^\omega$-actions on a probability space $(X,\mathcal{F},\mu)$, i.e.
\begin{itemize}
\item $S^a,T^a\colon X\to X$ are measurable maps for each $a\in\mathbb{A}^\omega$,
\item $S^\mathbf{0}=T^\mathbf{0}=\textup{identity}$,
\item $S^a S^b = S^{a+b}$, \,$T^a T^b = T^{a+b}$, \,$S^a T^b = T^b S^a$\, for $a,b\in\mathbb{A}^\omega$,
\item $\mu(S^{a}E) = \mu(E) = \mu(T^{a}E)$ for $a\in\mathbb{A}^\omega$ and $E\in\mathcal{F}$.
\end{itemize}
We also impose a technical assumption that $(X,\mathcal{F})$ is a standard Borel measurable space,
i.e.\@ $\mathcal{F}$ is a Borel $\sigma$-algebra of some separable complete metric space.

It is now natural to consider double averages of the form
\begin{equation}\label{eqdoublegroupaverages}
\textup{M}^{\mathbb{A}}_{n}(f,g) := \frac{1}{|\Phi_n|}\sum_{a\in \Phi_n}(f\circ S^a)(g\circ T^a)
\end{equation}
for some $\mathcal{F}$-measurable functions $f,g\colon X\to\mathbb{C}$ and a nonnegative integer $n$.
Here we write $|\Phi|$ for the number of elements of a finite set $\Phi$.
Slightly modified averages (replacing $T^a$ by $S^a T^a$) were already shown to converge by
Bergelson, McCutcheon, and Zhang \cite{BergelsonMcCutcheonZhang97}, even for a general countable amenable group.
Multiple ergodic averages for amenable groups were discussed by Zorin-Kranich \cite{ZorinKranich11}.
The case of several ``powers'' of the same measure preserving action of $(\mathbb{Z}/p\mathbb{Z})^\omega$, $p$ prime,
was studied in detail by Bergelson, Tao, and Ziegler \cite{BergelsonTaoZiegler13}, with more emphasis on identifying the limits.
Let us remark that all these results are either only qualitative or very weakly quantitative in nature,
in the sense that they do not provide any explicit bounds on the number of $\varepsilon$-jumps defined below.

Now we can formulate the main result of this paper.
\begin{theorem}\label{bilintheoremmain}
For every $2\leq p<\infty$ there exists a finite constant $C_{p}$
such that the following norm-variational estimate holds for any $\mathbb{A},S,T,\mu,f,g$ as above:
$$ \sup_{\substack{m,n_0,n_1,\ldots,n_m\\ 0\leq n_0<n_1<\cdots<n_m}}
\sum_{j=1}^{m} \big\|\textup{M}^{\mathbb{A}}_{n_{j-1}}(f,g)-\textup{M}^{\mathbb{A}}_{n_{j}}(f,g)\big\|^{p}_{\textup{L}^{p}(X,\mathcal{F},\mu)}
\leq C_{p}\|f\|^{p}_{\textup{L}^{2p}(X,\mathcal{F},\mu)} \|g\|^{p}_{\textup{L}^{2p}(X,\mathcal{F},\mu)} . $$
\end{theorem}

The number of \emph{$\varepsilon$-jumps} (or \emph{$\varepsilon$-fluctuations}) of a sequence $(M_n)_{n=0}^{\infty}$
in the space $\textup{L}^{p}(X,\mathcal{F},\mu)$ is defined to be the largest nonnegative integer $m$ for which there exist indices
$$ n_1<n'_1\leq n_2<n'_2\leq\cdots\leq n_m<n'_m $$
satisfying
$$ \|M_{n_j}-M_{n'_j}\|_{\textup{L}^p}\geq\varepsilon \text{ for } j=1,2,\ldots,m. $$
If such a finite $m$ does not exist, we conventionally set the number of $\varepsilon$-jumps to infinity.

Take $\varepsilon>0$, \,$2\leq p<\infty$, and two complex functions $f,g$ normalized by
$$ \|f\|_{\textup{L}^{2p}}=\|g\|_{\textup{L}^{2p}}=1 . $$
An immediate consequence of Theorem \ref{bilintheoremmain} is that $(\textup{M}^{\mathbb{A}}_{n}(f,g))_{n=0}^{\infty}$
has $O(\varepsilon^{-p})$ $\varepsilon$-jumps in $\textup{L}^p$.
On the other hand, if we take any $1\leq p<\infty$ and the functions such that
$$ \|f\|_{\textup{L}^\infty}=\|g\|_{\textup{L}^\infty}=1 , $$
then we can apply Theorem \ref{bilintheoremmain} with a sufficiently large exponent $q$ in the place of $p$
and use nestedness of the $\textup{L}^p$ spaces.
This way we conclude that the number of $\varepsilon$-jumps in the $\textup{L}^p$ norm of the sequence \eqref{eqdoublegroupaverages}
is $O(\varepsilon^{-N})$ for some $N>0$.
In particular, the number of $\varepsilon$-jumps is finite for any $\varepsilon>0$, so
we reprove the fact that the averages \eqref{eqdoublegroupaverages} converge in mean,
a qualitative result that can be obtained along the lines of the papers \cite{BergelsonMcCutcheonZhang97} or \cite{ConzeLesigne84}.

In order to prove Theorem \ref{bilintheoremmain} we will transfer it into a convenient estimate for functions on the real line
using the so-called \emph{Calder\'{o}n transference principle} \cite{Calderon68},
sometimes also named the \emph{reverse Furstenberg correspondence principle} \cite{Furstenberg81}; the details are in Section \ref{sectiontransfer}.
It is not absolutely necessary to transfer the estimate to the Cantor group model of the reals, as we do, and one could equality well work
with the group $\mathbb{A}^\omega$ itself, as a sort of Cantor group model of the integers.
However, working with the Cantor group will make the proof notationally simpler 
and the method from \cite{Kovac11}, \cite{Kovac12}, \cite{KovacThiele13} easier to apply,
all at the cost of a few extra paragraphs.
Moreover, Proposition \ref{bilinproprestated} from Section \ref{sectionanalysis} might also hint how to obtain an analogous result
on the more common averages \eqref{eqmultipleziaverages} for $r=2$,
but it is not yet clear how to resolve some of the technical difficulties; see the comments in Section \ref{sectionclosing}.
Even though one can express optimism towards the possibility of adapting the proof from $\mathbb{A}^\omega$ to $\mathbb{Z}$,
it is quite likely that the ideas presented here are not enough and that they should be combined with some novel time-frequency analysis.

Another source of motivation for quantifying norm convergence of multiple ergodic averages is the well-known problem
of establishing a.e.\@ convergence of averages \eqref{eqmultipleziaverages}, which is still open for $r\geq 2$ and for two general
commuting transformations $T_1$ and $T_2$.
The papers in this direction \cite{Bourgain90}, \cite{Demeter07}, \cite{DemeterThiele10}
are also quantitative in nature and attempt variants of the problem via real-analytic techniques.
However, we do not discuss pointwise convergence in this note.

\section{An estimate on the real line}
\label{sectionanalysis}

The main part of the proof of Theorem \ref{bilintheoremmain} will take place on the real line, as we have already announced.
Some of the arguments could equally well be performed directly on $\mathbb{A}^\omega$, but they would look artificial.

Let $\mathbb{R}_+$ denote the set of nonnegative real numbers, i.e.\@ $[0,\infty)$,
and let us also write $\mathbb{R}_{+}^{N}$ for a Cartesian power $[0,\infty)^N$. 
It will always be understood that $\mathbb{R}_+$ and $\mathbb{R}_{+}^{N}$ are equipped with the standard Borel $\sigma$-algebra and the Lebesgue measure.
We set
$$ d = |\mathbb{A}| = \,\text{the cardinality of }\mathbb{A} $$
and agree to identify $\mathbb{A}$ with $\{0,1,2,\ldots,d-1\}$ in any order, with the only constraint that
$0$ corresponds to the neutral element of $\mathbb{A}$.
The \emph{Cantor group structure} on $\mathbb{R}_+$ compatible with the previous setting is obtained in the following way.
Consider the group
$$ \mathbb{G} := \big\{(a_k)_{k\in\mathbb{Z}}\in\mathbb{A}^\mathbb{Z} \,:\,
\text{there exists } k_0\in\mathbb{Z} \text{ such that } a_k=0 \text{ for each } k>k_0\big\} $$
with coordinate-wise binary operation, with an ultrametric defined by
$$ \rho\big((a_k)_{k\in\mathbb{Z}},(b_k)_{k\in\mathbb{Z}}\big) :=
\left\{\begin{array}{cl} d^{k_0} & \text{ if } k_0 \text{ is the largest } k\in\mathbb{Z} \text{ such that } a_k\neq b_k, \\
0 & \text{ if } (a_k)_{k\in\mathbb{Z}}=(b_k)_{k\in\mathbb{Z}}, \end{array}\right. $$
and with the Haar measure $\lambda_\mathbb{G}$.
If we normalize $\lambda_\mathbb{G}$ properly, the two maps
$$ \begin{array}{ll}
\iota\colon \mathbb{G}\to\mathbb{R}_+, & \iota\colon (a_k)_{k\in\mathbb{Z}} \mapsto \sum_{k\in\mathbb{Z}} a_k d^k , \\[1mm]
\kappa\colon \mathbb{R}_+\to\mathbb{G}, & \kappa\colon t\mapsto (\lfloor d^{-k}t\rfloor\,\textup{mod}\,d)_{k\in\mathbb{Z}}
\end{array} $$
become a.e.-isomorphisms of measure spaces and are a.e.-inverse to each other.
The reason why we do not have exact correspondence between elements of $\mathbb{G}$ and elements of $\mathbb{R}_+$
is that countably many real numbers do not have unique representation in base $d$.
For $x,y\in\mathbb{R}_+$ we can set
$$ x\oplus y := \iota\big(\kappa(x)+\kappa(y)\big), \quad \ominus x := \iota\big(\!-\kappa(x)\big). $$
We use the signs $\oplus$ and $\ominus$ to distinguish from the usual operations on real numbers.
One can summarize the previous construction by saying that the set of nonnegative reals
is viewed (up to an a.e.\@ isomorphism) as a totally disconnected group with digit-wise binary operations,
$$ \Big(\sum_{k\in\mathbb{Z}}x_k d^k\Big)\oplus\Big(\sum_{k\in\mathbb{Z}}y_k d^k\Big)
:= \sum_{k\in\mathbb{Z}}(\underbrace{x_k+y_k}_{\text{in }\mathbb{A}}) d^k,
\quad \ominus\Big(\sum_{k\in\mathbb{Z}}x_k d^k\Big) := \Big(\sum_{k\in\mathbb{Z}}(\underbrace{-x_k}_{\text{in }\mathbb{A}}) d^k\Big) . $$
The case $\mathbb{A}=\mathbb{Z}/d\mathbb{Z}$ is prototypical and appears often in the literature,
especially when $d=2$, which corresponds to the \emph{Walsh model} \cite{Walsh23}.
The reader can refer to \cite{MuscaluTaoThiele03} for more details on the Cantor group structure on $\mathbb{R}_+$.

Take any two Borel-measurable functions $F,G\colon\mathbb{R}_{+}^{2}\to\mathbb{C}$ that are bounded and supported on an interval.
For any $k\in\mathbb{Z}$ we define a bilinear average
$$ \textup{A}_{k}(F,G)(x,y) := d^{k}\!\int_{[0,d^{-k})}\! F(x\oplus t,y) G(x,y\oplus t) \,dt . $$
Now we can formulate an analytic variant of Theorem \ref{bilintheoremmain}.

\begin{proposition}\label{bilinproprestated}
If $p\geq 2$ is a finite exponent, $m$ is a positive integer, and $k_0,k_1,\ldots,k_m$ are integers such that $k_0<k_1<\ldots<k_m$,
then the estimate
$$ \sum_{j=0}^{m-1} \big\|\textup{A}_{k_{j+1}}(F,G)-\textup{A}_{k_{j}}(F,G)\big\|^{p}_{\textup{L}^{p}(\mathbb{R}_{+}^{2})}
\,\leq\, C_{p}\, \|F\|^{p}_{\textup{L}^{2p}(\mathbb{R}_{+}^{2})} \|G\|^{p}_{\textup{L}^{2p}(\mathbb{R}_{+}^{2})} $$
holds with a constant $C_{p}$ depending only on $p$.
\end{proposition}

The rest of this section is devoted to the proof of Proposition \ref{bilinproprestated}.
We begin with a very elementary inequality for real numbers.

\begin{lemma}\label{diffestimatelemma}
For any $2\leq p<\infty$ there exists a constant $c_p>0$ such that the inequality
$$ |a|^p - |b|^p - p (a-b) b |b|^{p-2} \geq c_p |a-b|^p $$
holds for any $a,b\in\mathbb{R}$.
\end{lemma}

\begin{proof}
When either $b=0$ or $a=b$ or $p=2$, the inequality is obvious and becomes an identity with the constant $c_p$ simply being equal to $1$.
Therefore, in the following we assume that $a\neq b\neq 0$ and $p>2$.
Dividing by $|b|^p$ and substituting $t=\frac{a-b}{b}\neq 0$ the inequality transforms into
$$ |1+t|^p - 1 - pt \geq c_p |t|^p . $$

Define a function $\theta\colon\mathbb{R}\setminus\{0\}\to\mathbb{R}$ by the formula
$$ \theta(t) := \frac{|1+t|^p - 1 - pt}{|t|^p} . $$
Observe that $\theta(t)>0$ for every $t\neq 0$. Indeed, for $t>-1$, $t\neq 0$ this follows from Bernoulli's inequality,
$$ (1+t)^p > 1 + pt , $$
while for $t\leq -1$ this is entirely evident,
$$ |1+t|^p - 1 - pt \geq -1 + p > 0 . $$
Next, an easy application of L'H\^{o}pital's rule gives
$$ \lim_{t\to 0}\theta(t)=+\infty \quad\text{and}\quad \lim_{t\to\pm\infty}\theta(t)=1 . $$
Finally, note that $\theta$ is continuous on $\mathbb{R}\setminus\{0\}$.
From these properties it readily follows that $\theta$ is bounded from below by some constant $c_p>0$.
\end{proof}

Let us fix a positive integer $m$ and arbitrary integers $k_0<k_1<\ldots<k_m$.
It will be convenient to first work with integer values of $p$ and later invoke an interpolation argument.
Thus, suppose that $p\geq 2$ is a fixed integer.
Without loss of generality we can assume $F,G\geq 0$, as otherwise we split the functions into their real and complex, positive and negative parts.
Consequently, $\textup{A}_{k}(F,G)\geq 0$ for each $k$.

Denote
$$ \varphi_k := d^{k}\mathbf{1}_{[0,d^{-k})} $$
for any $k\in\mathbb{Z}$, where $\mathbf{1}_E$ denotes the characteristic function of a set $E$, so that
$$ \textup{A}_{k}(F,G)(x,y) := \int_{\mathbb{R}_+}\! F(x\oplus t,y) G(x,y\oplus t) \varphi_k(t) dt . $$
Applying Lemma \ref{diffestimatelemma} with
$$ a=\textup{A}_{k_{j+1}}(F,G)(x,y)\geq 0, \quad b=\textup{A}_{k_{j}}(F,G)(x,y)\geq 0 , $$
summing over $j=0,1,\ldots,m-1$, and telescoping the cancelling terms, we get
\begin{align*}
& \sum_{j=0}^{m-1}\Big|\textup{A}_{k_{j+1}}(F,G)(x,y)-\textup{A}_{k_{j}}(F,G)(x,y)\Big|^p \\
& \leq c_p^{-1} \bigg(\textup{A}_{k_{m}}(F,G)(x,y)^p - \textup{A}_{k_{0}}(F,G)(x,y)^p \\
& \qquad\quad - p \,\sum_{j=0}^{m-1} \Big(\textup{A}_{k_{j+1}}(F,G)(x,y)-\textup{A}_{k_{j}}(F,G)(x,y)\Big)\, \textup{A}_{k_{j}}(F,G)(x,y)^{p-1}\bigg) .
\end{align*}
Integrating in $(x,y)$ over $\mathbb{R}_{+}^{2}$ we conclude that the left hand side of the estimate in Proposition \ref{bilinproprestated}
is at most a constant depending on $p$ times
$$ \|\textup{A}_{k_{m}}(F,G)\|^{p}_{\textup{L}^{p}(\mathbb{R}_{+}^{2})} \,+\, |\Lambda(F,G)| , $$
where we have denoted
\begin{align*}
\Lambda(F,G) & := \sum_{j=0}^{m-1} \int_{\mathbb{R}_{+}^{2}}
\Big(\textup{A}_{k_{j+1}}(F,G)(x,y)-\textup{A}_{k_{j}}(F,G)(x,y)\Big)\, \textup{A}_{k_{j}}(F,G)(x,y)^{p-1} dx dy \\
& = \sum_{j=0}^{m-1} \int_{\mathbb{R}_{+}^{p+2}}
F(x\oplus t_1,y) \cdots F(x\oplus t_p,y) \,G(x,y\oplus t_1) \cdots G(x,y\oplus t_p) \\[-2mm]
& \qquad\qquad\qquad\ \big(\varphi_{k_{j+1}}(t_1)-\varphi_{k_j}(t_1)\big)\, \varphi_{k_j}(t_2) \cdots \varphi_{k_j}(t_p)
\,dx dy dt_1 \cdots dt_p .
\end{align*}
Since by the integral Minkowski and H\"{o}lder inequalities we have
\begin{align*}
\|\textup{A}_{k_{m}}(F,G)\|_{\textup{L}^{p}(\mathbb{R}_{+}^{2})}
& \leq \int_{\mathbb{R}_+} \|F(x\oplus t,y) \,G(x,y\oplus t)\|_{\textup{L}^{p}_{(x,y)}(\mathbb{R}_{+}^{2})} \,\varphi_{k_m}(t) dt \\
& \leq \int_{\mathbb{R}_+} \|F\|_{\textup{L}^{2p}(\mathbb{R}_{+}^{2})} \|G\|_{\textup{L}^{2p}(\mathbb{R}_{+}^{2})} \,\varphi_{k_m}(t) dt
= \|F\|_{\textup{L}^{2p}(\mathbb{R}_{+}^{2})} \|G\|_{\textup{L}^{2p}(\mathbb{R}_{+}^{2})} ,
\end{align*}
it is enough to show
\begin{equation}\label{eqsufficientlambda}
|\Lambda(F,G)| \leq \|F\|^{p}_{\textup{L}^{2p}(\mathbb{R}_{+}^{2})} \|G\|^{p}_{\textup{L}^{2p}(\mathbb{R}_{+}^{2})} .
\end{equation}

For the rest of the proof we need to introduce appropriate Haar functions.
Let
$$ \xi_0, \xi_1, \xi_2, \ldots, \xi_{d-1} \colon\mathbb{A}\to\textup{S}^{1}=\{z\in\mathbb{C}:|z|=1\} $$
be all characters of the abelian group $\mathbb{A}$, enumerated arbitrarily except for $\xi_0$ being constantly equal to $1$.
We first define $\mathbf{h}_{[0,1)}^{s}$ for any $s\in\{0,1,\ldots,d-1\}$ as
$$ \mathbf{h}_{[0,1)}^{s} := \sum_{i=0}^{d-1} \xi_s(i) \mathbf{1}_{[i/d,(i+1)/d)} $$
and then for any \emph{$d$-adic interval} $I=[d^{-k}l,d^{-k}(l+1))\subseteq\mathbb{R}_{+}$, $k,l\in\mathbb{Z}$, $l\geq 0$, we set
$$ \mathbf{h}_{I}^{s}(t) := \mathbf{h}_{[0,1)}^{s}(d^{k}t - l) . $$
An equivalent definition is
$$ \mathbf{h}_{I}^{s}(t) := \left\{\begin{array}{cl} \xi_s(a_{-k-1}) & \text{if } t\in I,\\
0 & \text{if } t\not\in I,\end{array}\right. $$
where $t=\sum_{j\in\mathbb{Z}}a_j d^j$ and $I$ is as before.
Note that in particular $|\mathbf{h}_{I}^{s}|=\mathbf{1}_{I}$ and also $\mathbf{h}_{I}^{0}=\mathbf{1}_{I}$.
The system $(\mathbf{h}_{I}^{s})_{I,s\neq 0}$ is a variant of the \emph{Haar system}.

The collection of all $d$-adic intervals inside $\mathbb{R}_+$ will be denoted by $\mathcal{I}$
and the length of some $I\in\mathcal{I}$ will simply be written as $|I|$.
If $I,J\in\mathcal{I}$ are intervals of the same length,
$$ I=[d^{-k}l_1,d^{-k}(l_1+1)), \ \ J=[d^{-k}l_2,d^{-k}(l_2+1)), \ \ k,l_1,l_2\in\mathbb{Z}, \ l_1,l_2\geq 0 , $$
then $I\oplus J$ and $\ominus I$ will denote the new intervals
$$ I\oplus J := \{x\oplus y : x\in I, y\in J\} = \big[d^{-k}(l_1\oplus l_2),d^{-k}((l_1\oplus l_2)+1)\big) $$
and
$$ \ominus I := \{\ominus x : x\in I\} = \big[d^{-k}(\ominus l_1),d^{-k}(\ominus l_1+1)\big) , $$
again of length $d^{-k}$.
Simply from the character property of $\xi_{s}$ we obtain
\begin{equation}\label{eqcharproperty}
\mathbf{h}_{I\oplus J}^{s}(x\oplus y) = \mathbf{h}_{I}^{s}(x) \mathbf{h}_{J}^{s}(y)
\ \text{ and }\ \mathbf{h}_{\ominus I}^{s}(\ominus x) = \overline{\mathbf{h}_{I}^{s}(x)}
\end{equation}
whenever $I,J\in\mathcal{I}$, $|I|=|J|$, \,$x\in I$, $y\in J$, \,$s\in\{0,1,\ldots,d-1\}$.
For $I\in\mathcal{I}$ and a nonnegative integer $N$ we will denote by $d^{N}I$ the unique interval $I'\in\mathcal{I}$ such that
$|I'|=d^{N}|I|$ and $I'\supseteq I$.
Each interval $I\in\mathcal{I}$ can be partitioned into $d$ intervals in $\mathcal{I}$ of size $d^{-1}|I|$,
which can be called the ``children'' of $I$.
In this terminology we can call $d^{N}I$ the ``ancestor'' of $I$ from $N$ generations back.

Observe that
$$ \mathbf{1}_{[0,d^{-k})} + \sum_{s=1}^{d-1}\mathbf{h}^{s}_{[0,d^{-k})}
= \sum_{s=0}^{d-1} \sum_{i=0}^{d-1} \xi_s(i) \mathbf{1}_{[d^{-k-1}i,d^{-k-1}(i+1))}
= d\,\mathbf{1}_{[0,d^{-k-1})} , $$
so we can decompose
$$ \varphi_{k+1} - \varphi_{k} = d^{k+1}\mathbf{1}_{[0,d^{-k-1})} - d^k\mathbf{1}_{[0,d^{-k})} = \sum_{s=1}^{d-1}d^k\mathbf{h}^{s}_{[0,d^{-k})}  $$
and then also write
\begin{equation}\label{eqtelescopingphi}
\varphi_{k_{j+1}} - \varphi_{k_j} = \sum_{r=k_j}^{k_{j+1}-1} \sum_{s=1}^{d-1} d^r\mathbf{h}^{s}_{[0,d^{-r})} .
\end{equation}
We substitute
$$ z_i = x \oplus y \oplus t_i, \quad \widetilde{F}(z,y) := F(z\ominus y,y), \quad \widetilde{G}(z,x) := G(x,z\ominus x) $$
in the last expression for $\Lambda(F,G)$, so that
$$ t_i=z_i\ominus x\ominus y, \quad F(x\oplus t_i,y):=\widetilde{F}(z_i,y), \quad G(x,y\oplus t_i):=\widetilde{G}(z_i,x) $$
and using \eqref{eqtelescopingphi} the form $\Lambda$ transforms into
\begin{align*}
& \Lambda(F,G) = \widetilde{\Lambda}(\widetilde{F},\widetilde{G}) \\
& = \sum_{j=0}^{m-1} \sum_{r=k_j}^{k_{j+1}-1} \sum_{s=1}^{d-1} \,\int_{\mathbb{R}_{+}^{p+2}}
\widetilde{F}(z_1,y) \cdots \widetilde{F}(z_p,y) \,\widetilde{G}(z_1,x) \cdots \widetilde{G}(z_p,x) \,d^{r+(p-1)k_j} \\
& \quad\ \mathbf{h}^{s}_{[0,d^{-r})}(z_1\!\ominus\! x\!\ominus\! y)
\mathbf{1}_{[0,d^{-k_j})}(z_2\!\ominus\! x\!\ominus\! y) \cdots \mathbf{1}_{[0,d^{-k_j})}(z_p\!\ominus\! x\!\ominus\! y)
\,dx dy dz_1 \cdots dz_p .
\end{align*}
Splitting the integrals in $x$ and $y$ by inserting $\mathbf{1}_I(x)\mathbf{1}_J(y)$ for any $d$-adic intervals $I,J$ of length $d^{-r}$
and using the character property \eqref{eqcharproperty} we obtain
\begin{align*}
\widetilde{\Lambda}(\widetilde{F},\widetilde{G}) & = \sum_{j=0}^{m-1} \sum_{r=k_j}^{k_{j+1}-1} \sum_{s=1}^{d-1}
\sum_{\substack{I,J\in\mathcal{I},\ |I|=|J|=d^{-r}\\ K:=I\oplus J,\ L:=d^{r-k_j}K}} d^{r+(p-1)k_j} \\
& \quad\ \int_{\mathbb{R}_{+}^{p+2}}
\widetilde{F}(z_1,y) \cdots \widetilde{F}(z_p,y) \,\widetilde{G}(z_1,x) \cdots \widetilde{G}(z_p,x) \\[-1mm]
& \qquad\qquad \overline{\mathbf{h}^{s}_{I}(x) \mathbf{h}^{s}_{J}(y)} \mathbf{h}^{s}_{K}(z_1)
\mathbf{1}_{L}(z_2)\cdots\mathbf{1}_{L}(z_p) \,dx dy dz_1 \cdots dz_p .
\end{align*}

The form $\widetilde{\Lambda}$ is reminiscent of \emph{entangled multilinear dyadic Calder\'{o}n-Zygmund operators}
studied in \cite{KovacThiele13}.
The basic idea for proving bounds for such operators is based on repeated applications of the Cauchy-Schwarz inequality and certain telescoping identities.
We adapt the same approach here.
Actually, for $\widetilde{\Lambda}$ this procedure will terminate after only one step and yield Estimate \eqref{eqsufficientlambda}.
We provide details of that argument, in order to keep the exposition self-contained and to explain the necessary modifications.

Taking absolute values,
\begin{align*}
|\widetilde{\Lambda}(\widetilde{F},\widetilde{G})| & \leq \sum_{j=0}^{m-1} \sum_{r=k_j}^{k_{j+1}-1} \sum_{s=1}^{d-1}
\sum_{\substack{I,J\in\mathcal{I},\ |I|=|J|=d^{-r}\\ K:=I\oplus J,\ L:=d^{r-k_j}K}} d^{r+(p-1)k_j} \\
& \quad\int_{\mathbb{R}_{+}^{p}}
\Big|\int_{\mathbb{R}_{+}}\widetilde{F}(z_1,y) \cdots \widetilde{F}(z_p,y) \overline{\mathbf{h}^{s}_{J}(y)} dy\Big|
\Big|\int_{\mathbb{R}_{+}}\widetilde{G}(z_1,x) \cdots \widetilde{G}(z_p,x) \overline{\mathbf{h}^{s}_{I}(x)} dx\Big| \\
& \qquad\qquad\qquad\qquad\qquad\qquad\qquad\qquad \mathbf{1}_{K}(z_1) \mathbf{1}_{L}(z_2)\cdots\mathbf{1}_{L}(z_p)
\,dz_1 dz_2 \cdots dz_p ,
\end{align*}
and then applying the Cauchy-Schwarz inequality we get
\begin{equation}\label{eqlambdacsb}
|\widetilde{\Lambda}(\widetilde{F},\widetilde{G})| \leq \Theta(\widetilde{F})^{1/2} \,\Theta(\widetilde{G})^{1/2},
\end{equation}
where we have denoted
\begin{align*}
\Theta(\widetilde{F}) & := \sum_{j=0}^{m-1} \sum_{r=k_j}^{k_{j+1}-1} \sum_{s=1}^{d-1}
\sum_{\substack{J,K\in\mathcal{I},\ |J|=|K|=d^{-r}\\ L:=d^{r-k_j}K}} d^{r+(p-1)k_j} \\
& \quad \int_{\mathbb{R}_{+}^{p}}
\Big|\int_{\mathbb{R}_{+}}\widetilde{F}(z_1,y) \cdots \widetilde{F}(z_p,y) \overline{\mathbf{h}^{s}_{J}(y)} dy\Big|^2
\mathbf{1}_{K}(z_1) \mathbf{1}_{L}(z_2)\cdots\mathbf{1}_{L}(z_p)
\,dz_1 dz_2 \cdots dz_p .
\end{align*}
Here we used the fact that if three intervals $I,J,K$ are related by the equality $I\oplus J=K$, then any two of them uniquely determine the third one.
Moreover, summing $\mathbf{1}_{K}(z_1)$ over all $K$ such that $d^{r-k_j}K=L$ for a fixed $L$ we get a slightly simpler expression,
\begin{align*}
\Theta(\widetilde{F}) & = \sum_{j=0}^{m-1} \sum_{r=k_j}^{k_{j+1}-1} \sum_{s=1}^{d-1}
\sum_{\substack{J,L\in\mathcal{I}\\ |J|=d^{-r},\,|L|=d^{-k_j}}} d^{r+(p-1)k_j} \\
& \quad\int_{\mathbb{R}_{+}^{p}}
\Big|\int_{\mathbb{R}_{+}}\widetilde{F}(z_1,y) \cdots \widetilde{F}(z_p,y) \overline{\mathbf{h}^{s}_{J}(y)} dy\Big|^2
\mathbf{1}_{L}(z_1) \mathbf{1}_{L}(z_2)\cdots\mathbf{1}_{L}(z_p)
\,dz_1 dz_2 \cdots dz_p .
\end{align*}
Because of $\|F\|_{\textup{L}^{2p}}=\|\widetilde{F}\|_{\textup{L}^{2p}}$,
$\|G\|_{\textup{L}^{2p}}=\|\widetilde{G}\|_{\textup{L}^{2p}}$, and \eqref{eqlambdacsb},
it is sufficient to verify
\begin{equation}\label{eqsufficienttheta}
\Theta(\widetilde{F}) \leq \|\widetilde{F}\|^{2p}_{\textup{L}^{2p}(\mathbb{R}_{+}^{2})}
\end{equation}
in order to prove \eqref{eqsufficientlambda}.

Let us expand out the last expression for $\Theta(\widetilde{F})$,
\begin{align*}
\Theta(\widetilde{F}) & = \sum_{j=0}^{m-1} \sum_{r=k_j}^{k_{j+1}-1} \sum_{s=1}^{d-1}
\sum_{\substack{J,L\in\mathcal{I}\\ |J|=d^{-r},\,|L|=d^{-k_j}}} d^{r+(p-1)k_j} \\
& \quad\ \int_{\mathbb{R}_{+}^{p+2}}
\widetilde{F}(z_1,y_1) \cdots \widetilde{F}(z_p,y_1) \,\widetilde{F}(z_1,y_2) \cdots \widetilde{F}(z_p,y_2) \\[-1mm]
& \qquad\qquad \mathbf{h}^{s}_{I}(y_1) \overline{\mathbf{h}^{s}_{J}(y_2)}
\mathbf{1}_{L}(z_1)\cdots\mathbf{1}_{L}(z_p) \,dy_1 dy_2 dz_1 \cdots dz_p .
\end{align*}
Using \eqref{eqcharproperty} and \eqref{eqtelescopingphi} once again we obtain
\begin{align*}
\sum_{r=k_j}^{k_{j+1}-1} \sum_{s=1}^{d-1}
\sum_{\substack{J\in\mathcal{I}\\ |J|=d^{-r}}} d^{r} \mathbf{h}^{s}_{J}(y_1) \overline{\mathbf{h}^{s}_{J}(y_2)}
& = \sum_{r=k_j}^{k_{j+1}-1} \sum_{s=1}^{d-1} d^{r} \mathbf{h}^{s}_{[0,d^{-r})}(y_1\!\ominus\!y_2) \\[-1mm]
& = \varphi_{k_{j+1}}(y_1\!\ominus\!y_2) - \varphi_{k_{j}}(y_1\!\ominus\!y_2) ,
\end{align*}
so that $\Theta(\widetilde{F})$ can be rewritten as
\begin{align*}
\Theta(\widetilde{F}) & = \sum_{j=0}^{m-1} \int_{\mathbb{R}_{+}^{p+2}}
\widetilde{F}(z_1,y_1) \cdots \widetilde{F}(z_p,y_1) \,\widetilde{F}(z_1,y_2) \cdots \widetilde{F}(z_p,y_2) \\[-1.5mm]
& \qquad\qquad\quad \big(\varphi_{k_{j+1}}(y_1\!\ominus\!y_2) - \varphi_{k_{j}}(y_1\!\ominus\!y_2)\big)
\,\vartheta_{k_j}(z_1,\ldots,z_p) \,dy_1 dy_2 dz_1 \cdots dz_p ,
\end{align*}
where
$$ \vartheta_{k}(z_1,\ldots,z_p) := \sum_{\substack{L\in\mathcal{I}\\ |L|=d^{-k}}} d^{(p-1)k}
\mathbf{1}_{L}(z_1) \mathbf{1}_{L}(z_2)\cdots\mathbf{1}_{L}(z_p) . $$
We introduce a ``complementary'' form
\begin{align*}
\Theta'(\widetilde{F}) & := \sum_{j=0}^{m-1} \int_{\mathbb{R}_{+}^{p+2}}
\widetilde{F}(z_1,y_1) \cdots \widetilde{F}(z_p,y_1) \,\widetilde{F}(z_1,y_2) \cdots \widetilde{F}(z_p,y_2) \\[-1.5mm]
& \qquad\qquad\quad \varphi_{k_{j+1}}(y_1\!\ominus\!y_2)
\,\big(\vartheta_{k_{j+1}}(z_1,\ldots,z_p)-\vartheta_{k_{j}}(z_1,\ldots,z_p)\big) \,dy_1 dy_2 dz_1 \cdots dz_p
\end{align*}
and also for any $k\in\mathbb{Z}$ we set
\begin{align*}
\Xi_{k}(\widetilde{F}) & = \int_{\mathbb{R}_{+}^{p+2}}
\widetilde{F}(z_1,y_1) \cdots \widetilde{F}(z_p,y_1) \,\widetilde{F}(z_1,y_2) \cdots \widetilde{F}(z_p,y_2) \\[-1mm]
& \qquad\qquad \varphi_{k}(y_1\!\ominus\!y_2) \,\vartheta_{k}(z_1,\ldots,z_p) \,dy_1 dy_2 dz_1 \cdots dz_p .
\end{align*}

The following identity and inequalities are similar to the ones used in \cite[\S 2]{Kovac11} and \cite[\S 3]{Kovac12}.
\begin{lemma}\label{auxineqlemma}
\begin{itemize}
\item[(a)] $\Theta(\widetilde{F}) + \Theta'(\widetilde{F}) = \Xi_{k_m}(\widetilde{F}) - \Xi_{k_0}(\widetilde{F})$,
\item[(b)] $\Theta'(\widetilde{F})\geq 0$,
\item[(c)] $0\leq \Xi_{k}(\widetilde{F})\leq \|\widetilde{F}\|^{2p}_{\textup{L}^{2p}(\mathbb{R}_{+}^{2})}$ for any $k\in\mathbb{Z}$.
\end{itemize}
\end{lemma}

\begin{proof}
Identity (a) follows immediately from the the summation by parts formula,
\begin{align*}
& \sum_{j=0}^{m-1}\big(\varphi_{k_{j+1}}(y_1\!\ominus\!y_2) - \varphi_{k_{j}}(y_1\!\ominus\!y_2)\big)\vartheta_{k_j}(z_1,\ldots,z_p) \\
& + \sum_{j=0}^{m-1}\varphi_{k_{j+1}}(y_1\!\ominus\!y_2)\big(\vartheta_{k_{j+1}}(z_1,\ldots,z_p)-\vartheta_{k_{j}}(z_1,\ldots,z_p)\big) \\
& = \varphi_{k_m}(y_1\!\ominus\!y_2)\vartheta_{k_m}(z_1,\ldots,z_p) - \varphi_{k_0}(y_1\!\ominus\!y_2)\vartheta_{k_0}(z_1,\ldots,z_p) .
\end{align*}

In order to prove (b) we rewrite $\Theta'(\widetilde{F})$ as
\begin{align*}
\Theta'(\widetilde{F}) & := \sum_{j=0}^{m-1} \int_{\mathbb{R}_{+}^{2}}
\bigg(d^{(p-1)k_{j+1}}\!\!\!\!\sum_{\substack{L'\in\mathcal{I}\\ |L'|=d^{-k_{j+1}}}}\!\!\Big(\int_{L'}\widetilde{F}(z,y_1)\widetilde{F}(z,y_2)dz\Big)^p \\
& \qquad\qquad\qquad
- d^{(p-1)k_j}\!\!\!\sum_{\substack{L\in\mathcal{I}\\ |L|=d^{-k_j}}}\!\Big(\int_{L}\widetilde{F}(z,y_1)\widetilde{F}(z,y_2)dz\Big)^p\bigg)
\varphi_{k_{j+1}}(y_1\!\ominus\!y_2)\,dy_1 dy_2 .
\end{align*}
For any $I\in\mathcal{I}$ we denote
$$ b_I := \int_{I}\widetilde{F}(z,y_1)\widetilde{F}(z,y_2)dz . $$
Recall that we are assuming $F\geq 0$ and hence also $\widetilde{F}\geq 0$, which implies $b_I\geq 0$.
It suffices to show that for any fixed interval $L\in\mathcal{I}$ we have
$$ N^{p-1}\sum_{i=1}^{N}b_{L_i}^p - \Big(\sum_{i=1}^{N}b_{L_i}\Big)^p \geq 0 , $$
where $N=d^{k_{j+1}-k_j}$ and $L_1,L_2,\ldots,L_N$ is the list of all intervals $L'\in\mathcal{I}$ such that $d^{k_{j+1}-k_j}L'=L$.
However, this is clearly a consequence of Jensen's inequality for the convex function $t\mapsto t^p$ on $[0,\infty)$.

Finally, we turn to part (c) and simplify $\Xi_{k}(\widetilde{F})$ as
\begin{align*}
\Xi_{k}(\widetilde{F}) & = \int_{\mathbb{R}_{+}^{2}} d^{(p-1)k}\!\!\sum_{\substack{L\in\mathcal{I}\\ |L|=d^{-k}}}
\!\!\Big(\int_{L}\widetilde{F}(z,y_1)\widetilde{F}(z,y_2)dz\Big)^p \,\varphi_{k}(y_1\!\ominus\!y_2) \,dy_1 dy_2 \\
& = \int_{\mathbb{R}_{+}^{2}} d^{(p-1)k}\!\!\sum_{\substack{L\in\mathcal{I}\\ |L|=d^{-k}}}
\!\!\Big(\int_{L}\widetilde{F}(z,y\oplus t)\widetilde{F}(z,y)dz\Big)^p \,\varphi_{k}(t) \,dy dt .
\end{align*}
By H\"{o}lder's inequality
\begin{align*}
\Xi_{k}(\widetilde{F}) & \leq \int_{\mathbb{R}_{+}} \sum_{\substack{L\in\mathcal{I}\\ |L|=d^{-k}}}\! |L|\,
\bigg(\int_{\mathbb{R}_{+}} \Big(\frac{1}{|L|}\int_{L}\widetilde{F}(z,y)^2 dz\Big)^p dy\bigg) \,\varphi_{k}(t) \,dt \\
& \leq \int_{\mathbb{R}_{+}} \bigg(\sum_{\substack{L\in\mathcal{I}\\ |L|=d^{-k}}}
\int_{L\times\mathbb{R}_{+}}\!\! \widetilde{F}(z,y)^{2p} dz dy\bigg) \,\varphi_{k}(t) \,dt
= \|\widetilde{F}\|^{2p}_{\textup{L}^{2p}(\mathbb{R}_{+}^{2})} ,
\end{align*}
which is exactly what we needed.
\end{proof}

Combining parts (a)--(c) of Lemma \ref{auxineqlemma} we establish \eqref{eqsufficienttheta} and thus also
complete the proof of Proposition \ref{bilinproprestated} in the particular case when $p$ is an integer.

In order to prove the result for a general $2\leq p<\infty$, we introduce a bilinear vector-valued operator
$$ T(F,G) := \big(\textup{A}_{k_{j+1}}(F,G)-\textup{A}_{k_{j}}(F,G)\big)_{0\leq j\leq m-1} $$
and the norm
$$ \|(V_j)_{0\leq j\leq m-1}\|_{\ell^{p}(\textup{L}^{p}(\mathbb{R}_{+}^{2}))}
:= \Big(\sum_{j=0}^{m-1}\|V_j\|_{\textup{L}^{p}(\mathbb{R}_{+}^{2})}^{p}\Big)^{1/p} . $$
Observe that Proposition \ref{bilinproprestated} is equivalent to boundedness of $T$ between Banach spaces
$$ \textup{L}^{2p}(\mathbb{R}_{+}^{2})\times\textup{L}^{2p}(\mathbb{R}_{+}^{2}) \to \ell^{p}(\textup{L}^{p}(\mathbb{R}_{+}^{2})) , $$
with an operator norm depending only on $p$.
It remains to interpolate between the integer values of $p$ using the multilinear complex interpolation of Banach-valued sequences,
discussed in \cite[\S 4.4 \& \S 5.6]{BerghLofstrom76}.

\section{The transference argument}
\label{sectiontransfer}

Even though the remaining part of the proof of Theorem \ref{bilintheoremmain} is an instance of a rather standard transference principle
and is easily adapted from \cite{ConzeLesigne84}, \cite{Furstenberg81}, or \cite{FurstenbergKatznelson78}, we prefer to include it for completeness.

Fix two functions $f,g$, an exponent $2\leq p<\infty$, and some non-negative integers $n_0<n_1<\cdots<n_m$.
The estimate in Theorem \ref{bilintheoremmain} can by homogeneity in $f$ and $g$ be equivalently written as
\begin{equation}\label{eqdehomogenized}
\int_{X}\sum_{j=1}^{m}\big|\textup{M}^{\mathbb{A}}_{n_{j-1}}(f,g)-\textup{M}^{\mathbb{A}}_{n_{j}}(f,g)\big|^p \,d\mu
\,\leq\, \int_{X}\frac{1}{2} C_{p} \big(|f|^{2p}+|g|^{2p}\big) \,d\mu .
\end{equation}
First, we claim that it is enough to prove \eqref{eqdehomogenized} in the particular case when
each set $E\in\mathcal{F}$ invariant under both $S$ and $T$ has either $\mu(E)=0$ or $\mu(E)=1$,
i.e.\@ when the action $(S^a T^b)_{(a,b)\in\mathbb{A}^\omega\times\mathbb{A}^\omega}$ of $\mathbb{A}^\omega\times\mathbb{A}^\omega$
on $(X,\mathcal{F},\mu)$ is ergodic.
Indeed, we can apply the ergodic decomposition by Varadarajan \cite{Varadarajan63} to obtain a probability space $(Y,\mathcal{G},\nu)$
and a family $(\mu_y)_{y\in Y}$ of probability measures on $(X,\mathcal{F})$ such that:
\begin{itemize}
\item for every $y\in Y$ the measure $\mu_y$ is invariant and ergodic with respect to the above action of $\mathbb{A}^\omega\times\mathbb{A}^\omega$,
\item $\int_{X}h\,d\mu = \int_{Y}\big(\int_{X}h\,d\mu_y\big)d\nu(y)$ holds for any $h\in\textup{L}^{1}(X,\mathcal{F},\mu)$.
\end{itemize}
Therefore, once we establish the ergodic case, we will have \eqref{eqdehomogenized} with each $\mu_y$ in place of $\mu$
and it will only remain to integrate over $Y$ with respect to the measure $\nu$.

Next, if we additionally assume that $(S^a T^b)_{(a,b)\in\mathbb{A}^\omega\times\mathbb{A}^\omega}$ is an ergodic action $(X,\mathcal{F},\mu)$,
then using the Lindenstrauss pointwise ergodic theorem \cite{Lindenstrauss01} for the F{\o}lner sequence $(\Phi_N\!\times\! \Phi_N)_{N=0}^{\infty}$ we obtain
$$ \lim_{N\to\infty} \frac{1}{|\Phi_N|^2}\sum_{a,b\in \Phi_N} h(S^a T^b x) = \int_{X} h \,d\mu $$
for every $h\in\textup{L}^{1}(X,\mathcal{F},\mu)$ and for $\mu$-a.e.\@ $x\in X$.
Thus, we can choose (a so-called \emph{generic point}) $x_0\in X$ such that the left hand side of \eqref{eqdehomogenized} can be expanded as
\begin{align*}
\lim_{N\to\infty} \frac{1}{|\Phi_N|^2}\sum_{a,b\in \Phi_N} \sum_{j=1}^{m}
\bigg|\frac{1}{|\Phi_{n_{j-1}}|}\sum_{c\in \Phi_{n_{j-1}}} f(S^{a+c}T^{b}x_0) g(S^{a}T^{b+c}x_0)\ \ & \\[-1mm]
- \frac{1}{|\Phi_{n_j}|}\sum_{c\in \Phi_{n_j}} f(S^{a+c}T^{b}x_0) g(S^{a}T^{b+c}x_0)\bigg|^p , &
\end{align*}
while the right hand side expands into
$$ \lim_{N\to\infty} \frac{1}{|\Phi_N|^2}\sum_{a,b\in \Phi_N} \frac{1}{2} C_{p} \big(|f(S^a T^b x_0)|^{2p}+|g(S^a T^b x_0)|^{2p}\big) . $$
We see that it suffices to fix a positive integer $N$ larger than $n_m$, define the functions
$F'$, $G'$, and $\textup{A}'_{n}(F',G')$ on $\mathbb{A}^\omega\times\mathbb{A}^\omega$ by
$$ F'(a,b):=\left\{\begin{array}{cl} f(S^a T^b x_0) & \text{if } a,b\in \Phi_N,\\ 0 & \text{otherwise}, \end{array}\right.
\quad G'(a,b):=\left\{\begin{array}{cl} g(S^a T^b x_0) & \text{if } a,b\in \Phi_N,\\ 0 & \text{otherwise}, \end{array}\right. $$
$$ \textup{A}'_{n}(F',G')(a,b) := \frac{1}{|\Phi_n|}\sum_{c\in\Phi_n} F'(a+c,b) G'(a,b+c) , $$
for each integer $0\leq n\leq N$, and then prove the inequality
\begin{equation}\label{eqfinaldiscrete}
\sum_{j=1}^{m} \big\|\textup{A}'_{n_{j-1}}(F',G')-\textup{A}'_{n_j}(F',G')\big\|^{p}_{\ell^{p}}
\leq \frac{1}{2} C_p \big(\|F'\|^{2p}_{\ell^{2p}}+\|G'\|^{2p}_{\ell^{2p}}\big) .
\end{equation}
Here the space $\ell^{p}$ is understood with respect to the counting measure on $\mathbb{A}^\omega\!\times\!\mathbb{A}^\omega$.

Finally, in order to derive \eqref{eqfinaldiscrete} from Proposition \ref{bilinproprestated} we view at $F'$ and $G'$
as functions on $\mathbb{Z}_{+}\!\times\!\mathbb{Z}_{+}$, where $\mathbb{Z}_+$ denotes the set of nonnegative integers,
and extend them to $\mathbb{R}_{+}\!\times\!\mathbb{R}_{+}$
in a way that they become constant on squares $[\alpha,\alpha+1)\times[\beta,\beta+1)$, $\alpha,\beta\in\mathbb{Z}_+$.
More precisely, restrictions of the maps $\iota,\kappa$ defined earlier,
$$ \begin{array}{ll}
\iota'\colon \mathbb{A}^\omega\to\mathbb{Z}_+, & \iota'\colon (a_k)_{k=0}^{\infty} \mapsto \sum_{k=0}^{\infty} a_k d^k , \\[1mm]
\kappa'\colon \mathbb{Z}_+\to\mathbb{A}^\omega, & \kappa'\colon t\mapsto (\lfloor d^{-k}t\rfloor\,\textup{mod}\,d)_{k=0}^{\infty} ,
\end{array} $$
realize a one-to-one and onto correspondence between elements of $\mathbb{A}^\omega$ and nonnegative integers.
If we define $F,G\colon\mathbb{R}_{+}^{2}\to\mathbb{C}$ by
$$ \begin{array}{l}
F(x,y) := \displaystyle\sum_{\alpha,\beta\in\mathbb{Z}_+} F'(\kappa'(\alpha),\kappa'(\beta))
\,\mathbf{1}_{[\alpha,\alpha+1)}(x) \,\mathbf{1}_{[\beta,\beta+1)}(y) , \\
G(x,y) := \displaystyle\sum_{\alpha,\beta\in\mathbb{Z}_+} G'(\kappa'(\alpha),\kappa'(\beta))
\,\mathbf{1}_{[\alpha,\alpha+1)}(x) \,\mathbf{1}_{[\beta,\beta+1)}(y) ,
\end{array} $$
then
\begin{align*}
\textup{A}_{-n}(F,G)(x,y) & = \frac{1}{d^n} \int_{[0,d^n)} \sum_{\substack{\alpha,\beta,\gamma\in\mathbb{Z}_{+}\\ 0\leq\gamma\leq d^n-1}}
F'\big(\kappa'(\alpha\oplus\gamma),\kappa'(\beta)\big) \,G'\big(\kappa'(\alpha),\kappa'(\beta\oplus\gamma)\big) \\[-3mm]
& \qquad\qquad\qquad\qquad\qquad\qquad
\mathbf{1}_{[\alpha,\alpha+1)}(x) \,\mathbf{1}_{[\beta,\beta+1)}(y) \,\mathbf{1}_{[\gamma,\gamma+1)}(t) \,dt \\[1mm]
& = \sum_{\alpha,\beta\in\mathbb{Z}_+} \textup{A}'_{n}(F',G')(\kappa'(\alpha),\kappa'(\beta))
\,\mathbf{1}_{[\alpha,\alpha+1)}(x) \,\mathbf{1}_{[\beta,\beta+1)}(y) ,
\end{align*}
so it remains to apply the estimate from Proposition \ref{bilinproprestated} with $k_j=-n_{m-j}$ for $j=0,1,\ldots,m$.

\section{Closing remarks}
\label{sectionclosing}

The idea of studying convergence of \eqref{eqmultipleziaverages} for $r=2$ via bilinear singular operators
was formulated by Demeter and Thiele \cite[\S 6]{DemeterThiele10}.
They suggest proving estimates for the operator
\begin{equation}\label{eqtriangularht}
T(F,G)(x,y) := \textup{p.v.}\int_{\mathbb{R}} F(x+t,y) G(x,y+t) \frac{dt}{t} ,
\end{equation}
for which still no positive or negative results are known at the time of writing; also see the remarks in \cite[\S 1]{KovacThiele13}.
The approach pursued in this paper differs slightly in the sense that the $\textup{L}^p$ norms are expanded out after applying
Lemma \ref{diffestimatelemma} and lead to ``less singular'' objects, closer to the level of operators studied in \cite{KovacThiele13}.

In order to answer the previously mentioned question by Avigad and Rute \cite{AvigadRute13} in the case $r=2$,
one would have to prove Proposition \ref{bilinproprestated} with the operation $\oplus$ replaced with the usual addition.
The main difficulty is then the lack of perfect localization in both time and frequency, which guarantees that the form $\Lambda$
has better cancellation properties than the original bilinear averages.
It is somewhat likely that certain ``error terms'' that appear would have to be controlled by objects similar to \eqref{eqtriangularht}.
However, already the results on entangled Calder\'{o}n-Zygmund operators in \cite{Kovac11} and \cite{KovacThiele13}
seem to be difficult to adapt to the Euclidean setting.

\section*{Acknowledgments}

I am grateful to Ciprian Demeter and Christoph Thiele for informing me about the question by Jeremy Avigad and Jason Rute \cite{AvigadRute13},
for pointing out the relevance of the techniques from \cite{Kovac11}, \cite{Kovac12}, and \cite{KovacThiele13} to the problem,
and for several earlier suggestions to reprove norm convergence of bilinear ergodic averages using methods from multilinear harmonic analysis.
I would also like to thank Fr\'{e}d\'{e}ric Bernicot for a useful discussion, which revealed the still unresolved obstructions
to an adaptation of the presented proof to bilinear averages \eqref{eqmultipleziaverages} for two commuting transformations.

\end{document}